\documentclass[a4,10pt]{amsart}
\usepackage{amssymb,amstext,amsmath,amscd,amsthm,amsfonts,enumerate,graphicx,latexsym}
\usepackage[usenames]{color}
\usepackage[all]{xy}

\newtheorem{thm}{Theorem}[section]
\newtheorem{thmm}{Theorem}

\newtheorem{cor}[thm]{Corollary}
\newtheorem{lem}[thm]{Lemma}
\newtheorem{prop}[thm]{Proposition}
\theoremstyle{definition}
\newtheorem{dfn}[thm]{Definition}
\newtheorem{rem}[thm]{Remark}
\newtheorem{ques}[thm]{Question}

\newtheorem{ex}[thm]{Example}
\newtheorem{nota}[thm]{Notation}

\newtheorem*{claim*}{Claim}
\theoremstyle{remark}

\numberwithin{equation}{thm}
\def\ocm{\mathrm{\Omega CM}}
\def\ul{\mathrm{Ul}}
\def\mod{\operatorname{mod}}
\def\cm{\mathrm{CM}}
\def\m{\mathfrak{m}}
\def\syz{\Omega}
\def\proj{\operatorname{proj}}
\def\inj{\operatorname{inj}}
\def\Hom{\operatorname{Hom}}
\def\Ext{\operatorname{Ext}}
\def\add{\operatorname{add}}
\def\ann{\operatorname{ann}}
\def\ind{\operatorname{ind}}
\def\s{\mathcal{S}}
\def\depth{\operatorname{depth}}
\def\xx{\boldsymbol{x}}
\def\H{\operatorname{H}}
\def\lhom{\operatorname{\underline{Hom}}}
\def\Tor{\operatorname{Tor}}
\def\tr{\mathrm{Tr}}
\def\a{\operatorname{ann^\mathsf{h}}}
\def\C{\mathbb{C}}
\def\End{\operatorname{End}}
\def\id{\mathsf{id}}
\def\Z{\mathbb{Z}}
\def\Cl{\operatorname{Cl}}

\def\p{\mathfrak{p}}
\def\ocms{\mathrm{\Omega CM}^\times}
\def\ge{\geqslant}

\def\le{\leqslant}
\def\leq{\leqslant}
\def\rank{\operatorname{rank}}
\def\iff{\Leftrightarrow}
\def\implies{\Rightarrow}

\def\edim{\operatorname{edim}}
\def\gr{\mathrm{gr}}
\def\ev{\operatorname{ev}}
\begin{document}
\allowdisplaybreaks
\title[Ulrich modules and minimal multiplicity]{Ulrich modules over Cohen--Macaulay local rings\\
with minimal multiplicity}
%\date{\today}
\author{Toshinori Kobayashi}
\address{Graduate School of Mathematics, Nagoya University, Furocho, Chikusaku, Nagoya, Aichi 464-8602, Japan}
\email{m16021z@math.nagoya-u.ac.jp}
\author{Ryo Takahashi}
\address{Graduate School of Mathematics, Nagoya University, Furocho, Chikusaku, Nagoya, Aichi 464-8602, Japan}
\email{takahashi@math.nagoya-u.ac.jp}
\urladdr{http://www.math.nagoya-u.ac.jp/~takahashi/}
\thanks{The second author was partly supported by JSPS Grant-in-Aid for Scientific Research 16H03923 and 16K05098}
%\dedicatory{Dedicated to Professor Ngo Viet Trung on the occasion of his sixtieth birthday}
\begin{abstract}
Let $R$ be a Cohen--Macaulay local ring.
In this paper we study the structure of Ulrich $R$-modules mainly in the case where $R$ has minimal multiplicity.
We explore generation of Ulrich $R$-modules, and clarify when the Ulrich $R$-modules are precisely the syzygies of maximal Cohen--Macaulay $R$-modules.
We also investigate the structure of Ulrich $R$-modules as an exact category.
\end{abstract}
\maketitle
%\tableofcontents
%%%%%%%%%%%%%%%%%%%%%%%%%%%%%%%%%%%%%%%%%%%%%%%%%%%%%%%%
\section*{Introduction}

The notion of an {\em Ulrich module}, which is also called a {\em maximally generated (maximal) Cohen--Macaulay module}, has first been studied by Ulrich \cite{U}, and widely investigated in both commutative algebra and algebraic geometry; see \cite{BHU,CH,CM,ulrich,rdp,HUB,KM,NY} for example.
A well-known conjecture asserts that Ulrich modules exist over any Cohen--Macaulay local ring $R$.
Even though the majority seem to believe that this conjecture does not hold true in full generality, a lot of partial (positive) solutions have been obtained so far.
One of them states that the conjecture holds whenever $R$ has minimal multiplicity (\cite{BHU}).
Thus, in this paper, mainly assuming that $R$ has minimal multiplicity, we are interested in what we can say about the structure of Ulrich $R$-modules.

We begin with exploring the number and generation of Ulrich modules.
The following theorem is a special case of our main results in this direction ($\syz$ denotes the first syzygy).

\begin{thmm}\label{A}
Let $(R,\m,k)$ be a $d$-dimensional complete Cohen--Macaulay local ring.
\begin{enumerate}[\rm(1)]
\item
Assume that $R$ is normal with $d=2$ and $k=\C$ and has minimal multiplicity.
If $R$ does not have a rational singularity, then there exist infinitely many indecomposable Ulrich $R$-modules.
\item
Suppose that $R$ has an isolated singularity.
Let $M,N$ be maximal Cohen--Macaulay $R$-modules with $\Ext_R^i(M,N)=0$ for all $1\le i\le d-1$.
If either $M$ or $N$ is Ulrich, then so is $\Hom_R(M,N)$.
\item
Let $\xx=x_1,\dots,x_d$ be a system of parameters of $R$ such that $\m^2=\xx\m$.
If $M$ is an Ulrich $R$-module, then so is $\syz(M/x_iM)$ for all $1\le i\le d$.
If one chooses $M$ to be indecomposable and not to be a direct summand of $\syz^dk$, then one finds an indecomposable Ulrich $R$-module not isomorphic to $M$ among the direct summands of the modules $\syz(M/x_iM)$.
\end{enumerate}
\end{thmm}

Next, we relate the Ulrich modules with the syzygies of maximal Cohen--Macaulay modules.
To state our result, we fix some notation.
Let $R$ be a Cohen--Macaulay local ring with canonical module $\omega$.
We denote by $\mod R$ the category of finitely generated $R$-modules, and by $\ul(R)$ and $\ocms(R)$ the full subcategories of Ulrich modules and first syzygies of maximal Cohen--Macaulay modules without free summands, respectively.
Denote by $(-)^\dag$ the canonical dual $\Hom_R(-,\omega)$.
Then $\ul(R)$ is closed under $(-)^\dag$, and contains $\ocms(R)$ if $R$ has minimal multiplicity.
The module $\syz^dk$ belongs to $\ocms(R)$, and hence $\syz^dk,(\syz^dk)^\dag$ belong to $\ul(R)$.
Thus it is natural to ask when the conditions in the theorem below hold, and we actually answer this.

\begin{thmm}\label{B}
Let $R$ be a $d$-dimensional singular Cohen--Macaulay local ring with residue field $k$ and canonical module $\omega$, and assume that $R$ has minimal multiplicity.
Consider the following conditions.
\begin{enumerate}[\quad\rm(1)]
\item
The equality $\ul(R)=\ocms(R)$ holds.
\item
The category $\ocms(R)$ is closed under $(-)^\dag$.
\item
The module $(\syz^dk)^\dag$ belongs to $\ocms(R)$.
\item
One has $\Tor_1(\tr(\syz^dk)^\dag,\omega)=0$.
\item
One has $\Ext_R^{d+1}(\tr(\syz^dk)^\dag,R)=0$ and $R$ is locally Gorenstein on the punctured spectrum.
\item
There is an epimorphism $\omega^{\oplus n}\to\syz^dk$ for some $n>0$.
\item
There is an isomorphism $\syz^dk\cong(\syz^dk)^\dag$.
\item
The local ring $R$ is almost Gorenstein.
\end{enumerate}
Then (1)--(6) are equivalent and (7) implies (1).
If $d>0$ and $k$ is infinite, then (1) implies (8).
If $d=1$ and $k$ is infinite, then (1)--(8) are equivalent.
If $R$ is complete normal with $d=2$ and $k=\C$, then (1)--(7) are equivalent unless $R$ has a cyclic quotient singularity.
\end{thmm}

Finally, we study the structure of the category $\ul(R)$ of Ulrich $R$-modules as an exact category in the sense of Quillen \cite{Q}.
We prove that if $R$ has minimal multiplicity, then $\ul(R)$ admits an exact structure with enough projective/injective objects.

\begin{thmm}\label{C}
Let $R$ be a $d$-dimensional Cohen--Macaulay local ring with residue field $k$ and canonical module, and assume that $R$ has minimal multiplicity.
Let $\s$ be the class of short exact sequences $0 \to L \to M \to N \to 0$ of $R$-modules with $L,M,N$ Ulrich.
%Set
%$$
%\s=\{\text{Short exact sequence of Ulrich $R$-modules}\}.
%$$
Then $(\ul(R),\s)$ is an exact category having enough projective objects and enough injective objects with $\proj\ul(R)=\add\syz^dk$ and $\inj\ul(R)=\add(\syz^dk)^\dag$.
\end{thmm}

The organization of this paper is as follows.
In Section \ref{dtc}, we deal with a question of Cuong on the number of indecomposable Ulrich modules.
We prove the first assertion of Theorem \ref{A} to answer this question in the negative.
In Section \ref{gen}, we consider how to generate Ulrich modules from given ones, and prove the second and third assertions of Theorem \ref{A}.
In Section \ref{ocmsul}, we compare Ulrich modules with syzygies of maximal Cohen--Macaulay modules, and prove Theorem \ref{B}; in fact, we obtain more equivalent and related conditions.
The final Section \ref{app} is devoted to giving applications of the results obtained in Section \ref{ocmsul}.
In this section we study the cases of dimension one and two, and exact structures of Ulrich modules, and prove the rest assertions of Theorem \ref{B} and Theorem \ref{C}.

%%%%%%%%%%%%%%%%%%%%%%%%%%%%%%%%%%%%%%%%%%%%%%%%%%%%%
\section*{Convention}
Throughout, let $(R,\m,k)$ be a Cohen--Macaulay local ring of Krull dimension $d$.
We assume that all modules are finitely generated and all subcategories are full.
A maximal Cohen--Macaulay module is simply called a Cohen--Macaulay module.
For an $R$-module $M$ we denote by $\syz M$ the first syzygy of $M$, that is, the kernel of the first differential map in the minimal free resolution of $M$.
Whenever $R$ admits a canonical module $\omega$, we denote by $(-)^\dag$ the canonical dual functor $\Hom_R(-,\omega)$.
For an $R$-module $M$ we denote by $e(M)$ and $\mu(M)$ the multiplicity and the minimal number of generators of $M$, respectively.

%%%%%%%%%%%%%%%%%%%%%%%%%%%%%%%%%%%%%%%%%%%%%%%%%%
\section{A question of Cuong}\label{dtc}

In this section, we consider a question raised by Cuong \cite{C} on the number of Ulrich modules over Cohen--Macaulay local rings with minimal multiplicity.
First of all, let us recall the definitions of an Ulrich module and minimal multiplicity.

\begin{dfn}
\begin{enumerate}[(1)]
\item
An $R$-module $M$ is called {\em Ulrich} if $M$ is Cohen--Macaulay with $e(M)=\mu(M)$.
\item
The ring $R$ is said to have {\em minimal multiplicity} if $e(R)=\edim R-\dim R+1$.
\end{enumerate}
\end{dfn}

An Ulrich module is also called a {\em maximally generated (maximal) Cohen--Macaulay} module.
There is always an inequality $e(R)\ge\edim R-\dim R+1$, from which the name of minimal multiplicity comes.
If $k$ is infinite, then $R$ has minimal multiplicity if and only if $\m^2=Q\m$ for some parameter ideal $Q$ of $R$.
See \cite[Exercise 4.6.14]{BH} for details of minimal multiplicity.

The following question has been raised by Cuong \cite{C}.

\begin{ques}[Cuong]\label{1}
If $R$ is non-Gorenstein and has minimal multiplicity, then are there only finitely many indecomposable Ulrich $R$-modules?
\end{ques}

To explore this question, we start by introducing notation, which is used throughout the paper.

\begin{nota}
We denote by $\mod R$ the category of finitely generated $R$-modules.
We use the following subcategories of $\mod R$:
\begin{align*}
\cm(R)&=\{M\in\mod R\mid\text{$M$ is Cohen--Macaulay}\},\\
\ul(R)&=\{M\in\cm(R)\mid\text{$M$ is Ulrich}\},\\
\ocm(R)&=\left\{M\in\cm(R)\,\bigg|\,
\begin{matrix}
\text{$M$ is the kernel of an epimorphism from a}\\
\text{free module to a Cohen--Macaulay module}
\end{matrix}
\right\},\\
\ocms(R)&=\{M\in\ocm(R)\mid\text{$M$ does not have a (nonzero) free summand}\}.
\end{align*}
\end{nota}

\begin{rem}
\begin{enumerate}[(1)]
\item
The subcategories $\cm(R),\ul(R),\ocm(R),\ocms(R)$ of $\mod R$ are closed under finite direct sums and direct summands.
\item
One has
$
\ocm(R)\cup\ul(R) \subseteq \cm(R) \subseteq \mod R.
$
\end{enumerate}
\end{rem}

Here we make a remark to reduce to the case where the residue field is infinite.

\begin{rem}\label{infin}
Consider the faithfully flat extension $S:=R[t]_{\m R[t]}$ of $R$.
Then we observe that:
\begin{enumerate}[(1)]
\item
If $X$ is a module in $\ocms(R)$, then $X\otimes_RS$ is in $\ocms(S)$.
\item
A module $Y$ is in $\ul(R)$ if and only if $Y\otimes_RS$ is in $\ul(S)$ (see \cite[Lemma 6.4.2]{HS}).
\end{enumerate}
The converse of (1) also holds true; we prove this in Corollary \ref{45}.

\end{rem}

If $R$ has minimal multiplicity, then all syzygies of Cohen--Macaulay modules are Ulrich:

\begin{prop}\label{2}
Suppose that $R$ has minimal multiplicity.
Then $\ocms(R)$ is contained in $\ul(R)$.
\end{prop}

\begin{proof}
By Remark \ref{infin} we may assume that $k$ is infinite.
Since $R$ has minimal multiplicity, we have $\m^2=Q\m$ for some parameter ideal $Q$ of $R$.
Let $M$ be a Cohen--Macaulay $R$-module.
There is a short exact sequence $0 \to \syz M \to R^{\oplus n} \to M \to 0$, where $n$ is the minimal number of generators of $M$.
Since $M$ is Cohen--Macaulay, taking the functor $R/Q\otimes_R-$ preserves the exactness; we get a short exact sequence
$$
0 \to \syz M/Q\syz M \xrightarrow{f} (R/Q)^{\oplus n} \to M/QM \to 0.
$$
The map $f$ factors through the inclusion map $X:=\m(R/Q)^{\oplus n}\to(R/Q)^{\oplus n}$, and hence there is an injection $\syz M/Q\syz M\to X$.
As $X$ is annihilated by $\m$, so is $\syz M/Q\syz M$.
Therefore $\m\syz M=Q\syz M$, which implies that $\syz M$ is Ulrich.
\end{proof}

As a direct consequence of \cite[Corollary 3.3]{ncr}, we obtain the following proposition.

\begin{prop}\label{3}
Let $R$ be a $2$-dimensional normal excellent henselian local ring with algebraically closed residue field of characteristic $0$.
Then there exist only finitely many indecomposable modules in $\ocm(R)$ if and only if $R$ has a rational singularity.
\end{prop}

Combining the above propositions yields the following result.

\begin{cor}\label{n}
Let $R$ be a $2$-dimensional normal excellent henselian local ring with algebraically closed residue field of characteristic $0$.
Suppose that $R$ has minimal multiplicity and does not have a rational singularity.
Then there exist infinitely many indecomposable Ulrich $R$-modules.
In particular, Quenstion \ref{1} has a negative answer.
\end{cor}

\begin{proof}
Proposition \ref{3} implies that $\ocm(R)$ contains infinitely many indecomposable modules, and so does $\ul(R)$ by Proposition \ref{2}.
\end{proof}

Here is an example of a non-Gorenstein ring satisfying the assumption of Corollary \ref{n}, which concludes that the question of Cuong is negative.

\begin{ex}
Let $B=\C[x,y,z,t]$ be a polynomial ring with $\deg x=\deg t=3$, $\deg y=5$ and $\deg z=7$.
Consider the $2\times3$-matrix $M=\left(\begin{smallmatrix}
x&y&z\\
y&z&x^3-t^3
\end{smallmatrix}\right)$ over $B$, and let $I$ be the ideal of $B$ generated by $2\times2$-minors of $M$.
Set $A=B/I$.
Then $A$ is a nonnegatively graded $\C$-algebra as $I$ is homogeneous.
By virtue of the Hilbert--Burch theorem (\cite[Theorem 1.4.17]{BH}), $A$ is a $2$-dimensional Cohen--Macaulay ring, and $x,t$ is a homogeneous system of parameters of $A$.
Directly calculating the Jacobian ideal $J$ of $A$, we can verify that $A/J$ is Artinian.
The Jacobian criterion implies that $A$ is a normal domain.
The quotient ring $A/tA$ is isomorphic to the numerical semigroup ring $\C[H]$ with $H=\langle3,5,7\rangle$.
Since this ring is not Gorenstein (as $H$ is not symmetric), neither is $A$.
Let $a(A)$ and $F(H)$ stand for the $a$-invariant of $A$ and the Frobenius number of $H$, respectively.
Then
$$
a(A)+3=a(A)+\deg(t)=a(A/tA)=F(H)=4,
$$
where the third equality follows from \cite[Theorem 3.1]{SR}.
Therefore we get $a(A)=1\not<0$, and $A$ does not have a rational singularity by the Flenner--Watanabe criterion (see \cite[Page 98]{LW}).

Let $A'$ be the localization of $A$ at $A_+$, and let $R$ be the completion of the local ring $A'$.
Then $R$ is a $2$-dimensional complete (hence excellent and henselian) normal non-Gorenstein local domain with residue field $\C$.
The maximal ideal $\m$ of $R$ satisfies $\m^2=(x,t)\m$, and thus $R$ has minimal multiplicity.
Having a rational singularity is preserved by localization since $A$ has an isolated singularity, while it is also preserved by completion.
Therefore $R$ does not have a rational singularity.
\end{ex}

We have seen that Question \ref{1} is not true in general.
However, in view of Corollary \ref{n}, we wonder if having a rational singularity is essential.
Thus, we pose a modified question.

\begin{ques}\label{4}
Let $R$ be a $2$-dimensional normal local ring with a rational singularity.
Then does $R$ have only finitely many indecomposable Ulrich modules?
\end{ques}

Proposition \ref{3} leads us to an even stronger question:

\begin{ques}\label{5}
If $\ocm(R)$ contains only finitely many indecomposable modules, then does $\ul(R)$ so?
\end{ques}

%%%%%%%%%%%%%%%%%%%%%%%%%%%%%%%%%%%%%%%%%%%%%%%%%%%%%%
\section{Generating Ulrich modules}\label{gen}

In this section, we study how to generate Ulrich modules from given ones.
First of all, we consider using the Hom functor to do it.

\begin{prop}\label{gotwy}
Let $M,N$ be Cohen--Macaulay $R$-modules.
Suppose that on the punctured spectrum of $R$ either $M$ is locally of finite projective dimension or $N$ is locally of finite injective dimension.
\begin{enumerate}[\rm(1)]
\item
$\Ext_R^i(M,N)=0$ for all $1\le i\le d-2$ if and only if $\Hom_R(M,N)$ is Cohen--Macaulay.
\item
Assume $\Ext_R^i(M,N)=0$ for all $1\le i\le d-1$.
If either $M$ or $N$ is Ulrich, then so is $\Hom_R(M,N)$.
\end{enumerate}
\end{prop}

\begin{proof}
(1) This follows from the proof of \cite[Proposition 2.5.1]{I}; in it the isolated singularity assumption is used only to have that the Ext modules have finite length.

(2) By (1), the module $\Hom_R(M,N)$ is Cohen--Macaulay.
We may assume that $k$ is infinite by Remark \ref{infin}(2), so that we can find a reduction $Q$ of $\m$ which is a parameter ideal of $R$.

First, let us consider the case where $N$ is Ulrich.
Take a minimal free resolution $F=(\cdots\to F_1\to F_0\to 0)$ of $M$.
Since $\Ext_R^i(M,N)=0$ for all $1\le i\le d-1$, the induced sequence
$$0 \to \Hom_R(M,N) \to \Hom_R(F_0,N) \xrightarrow{f} \cdots \to \Hom_R(F_{d-1},N)\to \Hom_R(\syz^dM,N)\to\Ext^d_R(M,N)\to0
$$
is exact.
Note that $\Ext_R^d(M,N)$ has finite length.
By the depth lemma, the image $L$ of the map $f$ is Cohen--Macaulay.
An exact sequence
$
0 \to \Hom_R(M,N) \to \Hom_R(F_0,N) \to L \to 0
$
is induced, and the application of the functor $-\otimes_RR/Q$ to this gives rise to an injection
$$
\Hom_R(M,N)\otimes_RR/Q \hookrightarrow \Hom_R(F_0,N)\otimes_RR/Q.
$$
Since $N$ is Ulrich, the module $\Hom_R(F_0,N)\otimes_RR/Q$ is annihilated by $\m$, and so is $\Hom_R(M,N)\otimes_RR/Q$.
Therefore $\Hom_R(M,N)$ is Ulrich.

Next, we consider the case where $M$ is Ulrich.
As $\xx$ is an $M$-sequence, there is a spectral sequence
$$
E_2^{pq}=\Ext_R^p(R/Q,\Ext_R^q(M,N))\ \Longrightarrow\ H^{p+q}=\Ext_R^{p+q}(M/QM,N).
$$
The fact that $\xx$ is an $R$-sequence implies $E_2^{pq}=0$ for $p>d$.
By assumption, $E_2^{pq}=0$ for $1\le q\le d-1$.
Hence an exact sequence
$
0 \to E_2^{d0} \to H^d \to E_2^{0d} \to 0
$
is induced.
Since $M/QM$ is annihilated by $\m$, so is $H^d=\Ext_R^d(M/QM,N)$, and so is $E_2^{d0}$.
Note that
$$
E_2^{d0}=\Ext_R^d(R/Q,\Hom_R(M,N))\cong\H^d(\xx,\Hom_R(M,N))\cong\Hom_R(M,N)\otimes_RR/Q,
$$
where $\H^*(\xx,-)$ stands for the Koszul cohomology.
If follows that $\m$ kills $\Hom_R(M,N)\otimes_RR/Q$, which implies that $\Hom_R(M,N)$ is Ulrich.
\end{proof}

As an immediate consequence of Proposition \ref{gotwy}(2), we obtain the following corollary, which is a special case of \cite[Theorem 5.1]{ulrich}.

\begin{cor}\label{5.6}
Suppose that $R$ admits a canonical module.
If $M\in\ul(R)$, then $M^\dag\in\ul(R)$.
\end{cor}

Next, we consider taking extensions of given Ulrich modules to obtain a new one.

\begin{prop}\label{6}
Let $Q$ be a parameter ideal of $R$ which is a reduction of $\m$.
Let $M,N$ be Ulrich $R$-modules, and take any element $a\in Q$.
Let
$
\sigma:0\to M\to E\to N\to0
$
be an exact sequence, and consider the multiplication
$
a\sigma: 0\to M\to X\to N\to 0
$
as an element of the $R$-module $\Ext_R^1(N,M)$.
Then $X$ is an Ulrich $R$-module.
\end{prop}

\begin{proof}
It follows from \cite[Theorem 1.1]{S} that the exact sequence
$$
a\sigma\otimes_RR/aR:\ 0\to M/aM\to X/aX\to N/aN\to0
$$
splits; we have an isomorphism $X/aX\cong M/aM\oplus N/aN$.
Applying the functor $-\otimes_{R/aR}R/Q$, we get an isomorphism $X/QX\cong M/QM\oplus N/QN$.
Since $M,N$ are Ulrich, the modules $M/QM,N/QN$ are $k$-vector spaces, and so is $X/QX$.
Hence $X$ is also Ulrich.
\end{proof}

As an application of the above proposition, we give a way to make an Ulrich module over a Cohen--Macaulay local ring with minimal multiplicity.

\begin{cor}\label{hp}
Let $Q$ be a parameter ideal of $R$ such that $\m^2=Q\m$.
Let $M$ be an Ulrich $R$-module.
Then for each $R$-regular element $a\in Q$, the syzygy $\syz(M/aM)$ is also an Ulrich $R$-module.
\end{cor}

\begin{proof}
There is an exact sequence
$
\sigma:0\to\syz M\to R^{\oplus n}\to M\to0,
$
where $n$ is a minimal number of generators of $M$.
We have a commutative diagram
$$
\xymatrix@R-1pc{
&&& 0 & 0\\
&&& M/aM\ar[u]\ar@{=}[r] & M/aM\ar[u]\\
\sigma: & 0\ar[r] & \syz M\ar[r] & R^{\oplus n}\ar[r]\ar[u] & M\ar[r]\ar[u] & 0\\
a\sigma: & 0\ar[r] & \syz M\ar[r]\ar@{=}[u] & X\ar[r]\ar[u] & M\ar[r]\ar[u]^a & 0\\
&&& 0\ar[u] & 0\ar[u]
}
$$
with exact rows and columns.
Since the minimal number of generators of $M/aM$ is equal to $n$, the middle column shows $X\cong\syz(M/aM)$.
Propositions \ref{1} and \ref{6} show that $X$ is Ulrich, and we are done.
\end{proof}

\begin{rem}\label{2.5}
In Corollary \ref{hp}, if the parameter ideal $Q$ annihilates the $R$-module $\Ext_R^1(M,\syz M)$, then we have $a\sigma=0$, and $\syz(M/aM)\cong M\oplus\syz M$.
Hence, in this case, the operation $M\mapsto\syz(M/aM)$ does not produce an essentially new Ulrich module.
\end{rem}

Next, we investigate the annihilators of Tor and Ext modules.

\begin{prop}
For an $R$-module $M$ one has
\begin{alignat*}{2}
&\ann_R\Ext_R^1(M,\syz M)
& &=\textstyle\bigcap_{i>0,\,N\in\mod R}\ann_R\Ext_R^i(M,N)\\
=\,&\ann_R\Tor^R_1(M,\tr M)
& &=\textstyle\bigcap_{i>0,\,N\in\mod R}\ann_R\Tor_i^R(M,N).
\end{alignat*}
\end{prop}

\begin{proof}
It is clear that
\begin{align*}
I:=\textstyle\bigcap_{i>0,\,N\in\mod R}\ann_R\Ext_R^i(M,N)
&\subseteq\ann_R\Ext_R^1(M,\syz M)\\
J:=\textstyle\bigcap_{i>0,\,N\in\mod R}\ann_R\Tor_i^R(M,N)
&\subseteq\ann_R\Tor^R_1(M,\tr M).
\end{align*}
It is enough to show that $\ann\Ext^1(M,\syz M)\cup\ann\Tor_1(M,\tr M)$ is contained in $I\cap J$.

(1) Take any element $a\in\ann_R\Ext_R^1(M,\syz M)$.
The proof of \cite[Lemma 2.14]{ua} shows that the multiplication map $(M\xrightarrow{a}M)$ factors through a free module, that is, $(M\xrightarrow{a}M)=(M\xrightarrow{f}F\xrightarrow{\pi}M)$ with $F$ free.
Hence, for all $i>0$ and $N\in\mod R$ we have commutative diagrams:
$$
\xymatrix@R-1pc@!C=30pt{
\Tor_i(M,N)\ar[rr]^a\ar[rd]_(.4){\Tor_i(f,N)\ } && \Tor_i(M,N)\\
& \Tor_i(F,N)\ar[ru]_(.6){\ \Tor_i(\pi,N)}
}
\qquad
\xymatrix@R-1pc@!C=30pt{
\Ext^i(M,N)\ar[rr]^a\ar[rd]_(.4){\Ext^i(\pi,N)\ } && \Ext^i(M,N)\\
& \Ext^i(F,N)\ar[ru]_(.6){\ \Ext^i(f,N)}
}
$$
As $\Tor_i(F,N)=\Ext^i(F,N)=0$, the element $a$ is in $I\cap J$.

(2) Let $b\in\ann_R\Tor_1^R(M,\tr M)$.
By \cite[Lemma (3.9)]{Y}, the element $b$ annihilates $\lhom_R(M,M)$.
Hence the map $b\cdot\id_M$, which is nothing but the multiplication map $(M\xrightarrow{b}M)$, factors through a free $R$-module.
Similarly to (1), we get $b$ is in $I\cap J$.
\end{proof}

\begin{dfn}
We denote by $\a M$ the ideal in the above proposition.
\end{dfn}

Note that $\a M=R$ if and only if $M$ is a free $R$-module.

For an $R$-module $M$ we denote by $\add M$ the subcategory of $\mod R$ consisting of direct summands of finite direct sums of copies of $M$.

With the notation of Remark \ref{2.5}, we are interested in when the operation $M\mapsto\syz(M/aM)$ actually gives rise to an essentially new Ulrich module.
The following result presents a possible way: if we choose an indecomposable Ulrich module $M$ that is not a direct summand of $\syz^dk$, then we find an indecomposable Ulrich module not isomorphic to $M$ among the direct summands of the modules $\syz(M/x_iM)$.

\begin{prop} \label{7b}
Suppose that $R$ is henselian.
Let $Q=(x_1,\dots,x_d)$ be a parameter ideal of $R$ which is a reduction of $\m$.
Let $M$ be an indecomposable Ulrich $R$-module.
If $M$ is a direct summand of $\syz(M/x_iM)$ for all $1\le i\le d$, then $M$ is a direct summand of $\syz^dk$.
\end{prop}

\begin{proof}
For all integer $1\le i\le d$ the module $\Ext_R^1(M,\syz M)$ is a direct summand of $\Ext_R^1(\syz(M/x_iM),\syz M)$.
The latter module is annihilated by $x_i$ since it is isomorphic to $\Ext_R^2(M/x_iM,\syz M)$.
Hence $Q$ is contained in $\ann_R\Ext_R^1(M,\syz M)=\a M$, and therefore $Q\Ext_R^{>0}(M,N)=0$ for all $N\in\mod R$.
It follows from \cite[Corollary 3.2(1)]{kos} that $M$ is a direct summand of $\syz^d(M/QM)$.
As $M$ is Ulrich, the module $M/QM$ is a $k$-vector space, and $\syz^d(M/QM)$ belongs to $\add(\syz^dk)$, whence so does $M$.
Since $R$ is henselian and $M$ is indecomposable, the Krull--Schmidt theorem implies that $M$ is a direct summand of $\syz^dk$.
\end{proof}

%%%%%%%%%%%%%%%%%%%%%%%%%%%%%%%%%
\section{Comparison of $\ul(R)$ with $\ocms(R)$}\label{ocmsul}

In this section, we study the relationship of the Ulrich $R$-modules with the syzygies of Cohen--Macaulay $R$-modules.
We begin with giving equivalent conditions for a given Cohen--Macaulay module to be a syzygy of a Cohen--Macaulay module, after stating an elementary lemma.

\begin{lem}\label{52}
Let $M,N$ be $R$-modules.
The evaluation map $\ev:M\otimes_R\Hom_R(M,N)\to N$ is surjective if and only if there exists an epimorphism $(f_1,\dots,f_n):M^{\oplus n}\to N$.
\end{lem}

\begin{proof}
The ``only if" part follows by taking an epimorphism $R^{\oplus n}\to\Hom_R(M,N)$ and tensoring $M$.
To show the ``if" part, pick any element $y\in N$.
Then we have $y=f_1(x_1)+\cdots+f_n(x_n)$ for some $x_1,\dots,x_n\in M$.
Therefore $y=\ev(\sum_{i=1}^nx_i\otimes f_i))$, and we are done.
\end{proof}

\begin{prop} \label{44}
Let $R$ be a Cohen--Macaulay local ring with canonical module $\omega$.
Then the following are equivalent for a Cohen--Macaulay $R$-module $M$.
\begin{enumerate}[\rm(1)]
\item
$M\in\ocm(R)$.
\item
$\lhom_R(M,\omega)=0$.
\item
There exists a surjective homomorphism $\omega^{\oplus n}\to \Hom_R(M,\omega)$.
\item
The natural homomorphism $\Phi:\omega\otimes_R \Hom_R(\omega, \Hom_R(M,\omega))\to \Hom_R(M,\omega)$ is surjective.
\item
$M$ is torsionless and $\tr\syz\tr M$ is Cohen--Macaulay.
\item
$\Ext^1_R(\tr M,R)=\Ext^1_R(\tr\syz\tr M,\omega)=0$.
\item
$\Tor_1^R(\tr M,\omega)=0$.
\end{enumerate}
\end{prop}

\begin{proof}
(1) $\implies$ (2):
By the assumption, there is an exact sequence $0\to M \to F \to N \to 0$ such that $N$ is Cohen--Macaulay and $F$ is free.
Take $f\in\Hom_R(M,\omega)$.
There is a commutative diagram
$$
\xymatrix@R-1pc{
0 \ar[r] & M \ar[d]^f \ar[r] & F \ar[d] \ar[r] & N \ar@{=}[d] \ar[r] & 0\\
0 \ar[r] & \omega \ar[r] & W \ar[r] & N \ar[r] &0
}
$$
with exact rows.
Since $N$ is Cohen--Macaulay, we have $\Ext_R^1(N,\omega)=0$.
Hence the second row splits, and $f$ factors through $F$.
This shows $\lhom_R(M,\omega)=0$.

(2) $\implies$ (1):
There is an exact sequence $0\to M \xrightarrow{f} \omega^{\oplus m}\to N \to 0$ such that $N$ is Cohen--Macaulay.
Since $\lhom_R(M,\omega^{\oplus m})=\lhom_R(M,\omega)^{\oplus m}=0$, there are a free $R$-module $F$, homomorphisms $g:M\to F$ and $h:F \to \omega^{\oplus m}$ such that $f=hg$.
We get a commutative diagram
$$
\xymatrix@R-1pc{
0 \ar[r] & M \ar@{=}[d] \ar[r]^{g} & F \ar[d]^h \ar[r] & L \ar[d] \ar[r] & 0\\
0 \ar[r] & M \ar[r]^f & \omega^{\oplus m} \ar[r] & N \ar[r] & 0
}
$$
with exact rows.
The secound square is a pullback-pushout diagram, which gives an exact sequence $0 \to F \to L\oplus \omega^{\oplus m} \to N \to 0$.
This shows that $L$ is Cohen--Macaulay, and hence $M\in\ocm(R)$.

(2) $\iff$ (7):
This equivalence follows from \cite[Lemma (3.9)]{Y}.

(1) $\implies$ (3):
Let $0\to M \to R^{\oplus n} \to N \to 0$ be an exact sequence with $F$ free.
Applying $(-)^\dag$, we have an exact sequence $0 \to N^\dag \to \omega^{\oplus n} \to M^\dag \to 0$.

(3) $\implies$ (1):
There is an exact sequence $0\to K \to \omega^{\oplus n} \to M^\dag \to 0$.
It is seen that $K$ is Cohen--Macaulay.
Taking $(-)^\dag$ gives an exact sequence $0 \to M \to R^{\oplus n} \to K^\dag \to 0$, which shows $M\in\ocm(R)$.

(3) $\iff$ (4):
This follows from Lemma \ref{52}.

(5) $\iff$ (6):
The module $\tr\syz\tr M$ is Cohen--Macaulay if and only if $\Ext^i_R(\tr\syz\tr M,\omega)=0$ for all $i>0$.
One has $\Ext^1_R(\tr M,R)=0$ if and only if $M$ is torsionless, if and only if $M\cong\syz\tr\syz\tr M$ up to free summands; see \cite[Theorem (2.17)]{AB}.
Hence $\Ext^i_R(\tr\syz\tr M,\omega)=\Ext^{i-1}_R(M,\omega)=0$ for all $i>1$.

(1) $\iff$ (5):
This equivalence follows from \cite[Lemma 2.5]{Ko2} and its proof.
\end{proof}

\begin{rem}\label{46}
The equivalence (1) $\iff$ (5) in Proposition \ref{44} holds without the assumption that $R$ admits a canonical module.
Indeed, its proof does not use the existence of a canonical module.
\end{rem}

The property of being a syzygy of a Cohen--Macaulay module (without free summand) is preserved under faithfully flat extension.

\begin{cor}\label{45}
Let $R \to S$ be a faithfully flat homomorphism of Cohen--Macaulay local rings.
Let $M$ be a Cohen--Macaulay $R$-module.
Then $M\in \ocms(R)$ if and only if $M\otimes_R S \in \ocms(S)$.
\end{cor}

\begin{proof}
Using Remark \ref{46}, we see that $M\in\ocm(R)$ if and only if $\Ext_R^1(\tr_RM,R)=0$ and $\tr_R\syz_R\tr_RM$ is Cohen--Macaulay.
Also, $M$ has a nonzero $R$-free summand if and only if the evaluation map $M\otimes_R\Hom_R(M,R)\to R$ is surjective by Lemma \ref{52}.
Since the latter conditions are both preserved under faithfully flat extension, they are equivalent to saying that $M\otimes_RS\in\ocm(S)$ and that $M\otimes_R S$ has a nonzero $S$-free summand, respectively.
Now the assertion follows.
\end{proof}

Next we state and prove a couple of lemmas.
The first one concerns Ulrich modules and syzygies of Cohen--Macaulay modules with respect to short exact sequences.

\begin{lem}\label{7.9}
Let $0\to L \to M \to N \to 0$ be an exact sequence of $R$-modules.
\begin{enumerate}[\rm(1)]
\item
If $L,M,N$ are in $\ul(R)$, then the equality $\mu(M)=\mu(L)+\mu(N)$ holds.
\item
Suppose that $L,M,N$ are in $\cm(R)$.
Then:\\
{\rm(a)} If $M$ is in $\ul(R)$, then so are $L$ and $N$.\qquad
{\rm(b)} If $M$ is in $\ocms(R)$, then so is $L$.
\end{enumerate}
\end{lem}

\begin{proof}
(1) We have $\mu(M)=e(M)=e(L)+e(N)=\mu(L)+\mu(N)$.

(2) Assertion (a) follows by \cite[Proposition (1.4)]{BHU}.
Let us show (b).
As $M$ is in $\ocms(R)$, there is an exact sequence $0 \to M \xrightarrow{\beta} R^{\oplus a} \xrightarrow{\gamma} C \to 0$ with $C$ Cohen--Macaulay.
As $M$ has no free summand, $\gamma$ is a minimal homomorphism.
In particular, $\mu(C)=a$.
The pushout of $\beta$ and $\gamma$ gives a commutative diagram
$$
\xymatrix@R-1pc{
& & 0 \ar[d] & 0 \ar[d] &\\
0 \ar[r] & L \ar@{=}[d] \ar[r] & M \ar[d]^\beta \ar[r] & N \ar[d] \ar[r] & 0\\
0 \ar[r] & L \ar[r] & R^{\oplus a} \ar[d]^\gamma \ar[r]^\delta & D \ar[d] \ar[r] & 0\\
& & C \ar[d] \ar@{=}[r] & C \ar[d] &\\
& & 0 & 0 &
}
$$
with exact rows and columns.
We see that $a=\mu(C)\leq\mu(D)\leq a$, which implies that $\delta$ is a minimal homomorphism.
Hence $L=\syz D\in \ocms(R)$.
\end{proof}

The following lemma is used to reduce to the case of a lower dimensional ring.

\begin{lem} \label{7a}
Let $Q=(x_1,\dots,x_d)$ be a parameter ideal of $R$ that is a reduction of $\m$.
Let $M$ be a Cohen--Macaulay $R$-module.
Then $M$ is an Ulrich $R$-module if and only if $M/x_iM$ is an Ulrich $R/x_iR$-module.
\end{lem}

\begin{proof}
Note that $Q/x_iR$ is a reduction of $\m/x_iR$.
We see that $(\m/x_iR)(M/x_iM)=(Q/x_iR)(M/x_iM)$ if and only if $\m M=QM$.
Thus the assertion holds.
\end{proof}

Now we explore syzygies of the residue field of a Cohen--Macaulay local ring with minimal multiplicity.

\begin{lem} \label{8}
Assume that $R$ is singular and has minimal multiplicity.
\begin{enumerate}[\rm(1)]
\item
One has $\syz^d_R k\in \ocms(R)$.
In particular, $\syz^d_R k$ is an Ulrich $R$-module.
\item
There is an isomorphism $\syz^{d+1}_R k\cong(\syz^d_R k)^{\oplus n}$ for some $n\ge0$.
\item
Let $Q=(x_1,\dots,x_d)$ be a parameter ideal of $R$ with $\m^2=Q\m$, and suppose that $d\ge1$.
Then $\syz^1_R(\syz^i_{R/(x_1)}k)\cong\syz^{i+1}_R k$ for all $i\ge0$.
In particular, $\syz^1_R(\syz^{d-1}_{R/(x_1)}k)\cong\syz^d_R k$.
\item 
For each $M\in\ul(R)$ there exists a surjective homomorphism $(\syz^d_R k)^{\oplus n} \to M$ for some $n\ge0$.
\end{enumerate}
\end{lem}

\begin{proof}
(1)(2) We may assume that $k$ is infinite; see Remark \ref{infin}.
So we find a parameter ideal $Q=(x_1,\dots,x_d)$ of $R$ with $\m^2=Q\m$.
The module $\m/Q$ is a $k$-vector space, and there is an exact sequence $0\to k^{\oplus n} \to R/Q \to k \to 0$.
Taking the $d$th syzygies gives an exact sequence
$$
0\to (\syz^dk)^{\oplus n} \to R^{\oplus t} \to \syz^d k \to 0.
$$
Since $\syz^d k$ has no free summand by \cite[Theorem 1.1]{syz2}, we obtain $\syz^d k\in\ocms(R)$ and $(\syz^d k)^{\oplus n}\cong \syz^{d+1} k$.
The last assertion of (1) follows from this and Proposition \ref{2}.

(3) Set $x=x_1$.
We show that $\syz(\syz_{R/xR}^ik)\cong\syz^{i+1}k$ for all $i\ge0$.
We may assume $i\ge1$; note then that $x$ is $\syz^ik$-regular.
By \cite[Corollary 5.3]{syz2} we have an isomorphism $\syz^ik/x\syz^ik\cong\syz_{R/xR}^ik\oplus\syz_{R/xR}^{i-1}k$.
Hence
\begin{equation}\label{sss}
\syz^ik\oplus\syz^{i+1}k\cong\syz(\syz^ik/x\syz^ik)\cong\syz(\syz_{R/xR}^ik)\oplus\syz(\syz_{R/xR}^{i-1}k),
\end{equation}
where the first isomorphism follows from the proof of Corollary \ref{hp}.
There is an exact sequence
$
0 \to \syz_{R/xR}^ik \to (R/xR)^{\oplus a_{i-1}} \to \cdots\to(R/xR)^{\oplus a_0} \to k \to 0
$
of $R/xR$-modules, which gives an exact sequence
$$
0 \to \syz(\syz_{R/xR}^ik) \to R^{\oplus b_{i-1}} \to \cdots \to R^{\oplus b_0} \to \syz k \to 0
$$
of $R$-modules.
This shows $\syz(\syz_{R/xR}^ik)\cong\syz^{i+1}k\oplus R^{\oplus u}$ for some $u\ge0$, and similarly we have an isomorphism $\syz(\syz_{R/xR}^{i-1}k)\cong\syz^ik\oplus R^{\oplus v}$ for some $v\ge0$.
Substituting these in \eqref{sss}, we see $u=v=0$ and obtain an isomorphism $\syz(\syz_{R/xR}^ik)\cong\syz^{i+1}k$.

(4) According to Lemma \ref{52} and Remark \ref{infin}, we may assume that $k$ is infinite.
Take a parameter ideal $Q=(x_1,\dots,x_d)$ of $R$ with $\m^2=Q\m$.
We prove this by induction on $d$.
If $d=0$, then $M$ is a $k$-vector space, and there is nothing to show.
Assume $d\ge1$ and set $x=x_1$.
Clearly, $R/xR$ has minimal multiplicity.
By Lemma \ref{7a}, $M/xM$ is an Ulrich $R/xR$-module.
The induction hypothesis gives an exact sequence
$
0 \to L \to (\syz^{d-1}_{R/xR} k)^{\oplus n} \to M/xM \to 0
$
of $R/xR$-modules.
Lemma \ref{7.9}(2) shows that $L$ is also an Ulrich $R/xR$-module, while Lemma \ref{7.9}(1) implies
$$
\mu_{R/xR}(L)+\mu_{R/xR}(M/xM)=\mu_{R/xR}((\syz^{d-1}_{R/xR} k)^{\oplus n}).
$$
Note that $\mu_R(X)=\mu_{R/xR}(X)$ for an $R/xR$-module $X$.
Thus, taking the first syzygies over $R$, we get an exact sequence of $R$-modules:
$$
0 \to \syz L \to \syz(\syz^{d-1}_{R/xR} k)^{\oplus n} \to \syz(M/xM) \to 0.
$$
From the proof of Corollary \ref{hp} we see that there is an exact sequence $0\to\syz M\to\syz(M/xM)\to M\to0$, while $\syz(\syz^{d-1}_{R/xR} k)$ is isomorphic to $\syz^dk$ by (3).
Consequently, we obtain a surjection $(\syz^dk)^{\oplus n}\to M$.
\end{proof}

We have reached the stage to state and prove the main result of this section.

\begin{thm}\label{five}
Let $R$ be a $d$-dimensional Cohen--Macaulay local ring with residue field $k$ and canonical module $\omega$.
Suppose that $R$ has minimal multiplicity.
Then the following are equivalent.
\begin{enumerate}[\rm(1)]
\item
The equality $\ocms(R)=\ul(R)$ holds.
\item
For an exact sequence $M \to N \to 0$ in $\cm(R)$, if $M\in\ocms(R)$, then $N\in\ocms(R)$.
\item
The category $\ocms(R)$ is closed under $(-)^\dag$.
\item
The module $(\syz^dk)^\dag$ belongs to $\ocms(R)$.\quad
{\rm(4')} The module $(\syz^dk)^\dag$ belongs to $\ocm(R)$.
\item
One has $\lhom_R((\syz^dk)^\dag,\omega)=0$.
\item
One has $\Tor_1^R(\tr((\syz^dk)^\dag),\omega)=0$.
\item
One has $\Ext_R^{d+1}(\tr((\syz^dk)^\dag),R)=0$ and $R$ is locally Gorenstein on the punctured spectrum.
\item
The natural homomorphism $\omega\otimes_R\Hom_R(\omega,\syz^dk)\to\syz^dk$ is surjective.
\item
There exists a surjective homomorphism $\omega^{\oplus n}\to\syz^dk$.
\end{enumerate}
If $d$ is positive, $k$ is infinite and one of the above nine conditions holds, then $R$ is almost Gorenstein.
\end{thm}

\begin{proof}
(1) $\implies$ (2):
This follows from Lemma \ref{7.9}(2).

(2) $\implies$ (3):
Let $M$ be an $R$-module in $\ocms(R)$.
Then $M\in \ul(R)$ by Proposition \ref{2}, and hence $M^\dag \in \ul(R)$ by Corollary \ref{5.6}.
It follows from Lemma \ref{8}(4) that there is a surjection $(\syz^d k)^{\oplus n}\to M^\dag$.
Since $(\syz^dk)^{\oplus n}$ is in $\ocms(R)$ by Lemma \ref{8}(1), the module $M^\dag$ is also in $\ocms(R)$. 

(3) $\implies$ (4):
Lemma \ref{8}(1) says that $\syz^d k$ is in $\ocms(R)$, and so is $(\syz^d k)^\dag$ by assumption.

(4) $\implies$ (1):
The inclusion $\ocms(R)\subseteq\ul(R)$ follows from Proposition \ref{2}.
Take any module $M$ in $\ul(R)$.
Then $M^\dag$ is also in $\ul(R)$ by Corollary \ref{5.6}.
Using Lemma \ref{8}(4), we get an exact sequence $0\to X \to (\syz^d k)^{\oplus n} \to M^\dag \to 0$ of Cohen--Macaulay modules, which induces an exact sequence $0\to M \to (\syz^dk)^{\dag \oplus n} \to X^\dag \to 0$.
The assumption and Lemma \ref{7.9}(2) imply that $M$ is in $\ocms(R)$.

(4) $\iff$ (4'):
As $R$ is singular, by \cite[Corollary 4.4]{syz2} the module $(\syz^dk)^\dag$ does not have a free summand.

(4') $\iff$ (5) $\iff$ (6) $\iff$ (8) $\iff$ (9):
These equivalences follow from Proposition \ref{44}.

(4') $\iff$ (7):
We claim that, under the assumption that $R$ is locally Gorenstein on the punctured spectrum, $(\syz^dk)^\dag\in\ocm(R)$ if and only if $\Ext_R^{d+1}(\tr((\syz^dk)^\dag),R)=0$.
In fact, since $(\syz^dk)^\dag$ is Cohen--Macaulay, it satisfies Serre's condition $(S_d)$.
Therefore it is $d$-torsionfree, that is, $\Ext_R^i(\tr((\syz^dk)^\dag),R)=0$ for all $1\le i\le d$; see \cite[Theorem 2.3]{tc}.
Hence, $\Ext_R^{d+1}(\tr((\syz^dk)^\dag),R)=0$ if and only if $(\syz^dk)^\dag$ is $(d+1)$-torsionfree, if and only if it belongs to $\ocm(R)$ by \cite[Theorem 2.3]{tc} again.
Thus the claim follows.

According to this claim, it suffices to prove that if (4') holds, then $R$ is locally Gorenstein on the punctured spectrum.
For this, pick any nonmaximal prime ideal $\p$ of $R$.
There are exact sequences
$$
0 \to \syz^dk \to R^{\oplus a_{d-1}} \to\cdots\to R^{\oplus a_0}\to k\to0,\qquad
0 \to(\syz^dk)_\p\to R_\p^{\oplus a_{d-1}} \to\cdots\to R_\p^{\oplus a_0}\to 0.
$$
We observe that $(\syz^dk)_\p$ is a free $R_\p$-module with $\rank_{R_\p}((\syz^dk)_\p)=\sum_{i=0}^{d-1}(-1)^ia_{d-1-i}=\rank_R(\syz^dk)$.
The module $\syz^dk$ has positive rank as it is torsionfree, and we see that $(\syz^dk)_\p$ is a nonzero free $R_\p$-module.
Since we have already shown that (4') implies (9), there is a surjection $\omega^{\oplus n}\to\syz^dk$.
Localizing this at $\p$, we see that $\omega_\p^{\oplus n}$ has an $R_\p$-free summand, which implies that the $R_\p$-module $R_\p$ has finite injective dimension.
Thus $R_\p$ is Gorenstein.

So far we have proved the equivalence of the conditions (1)--(9).
It remains to prove that $R$ is almost Gorenstein under the assumption that $d$ is positive, $k$ is infinite and (1)--(9) all hold.
We use induction on $d$.

Let $d=1$.
Let $Q$ be the total quotient ring of $R$, and set $E=\End_R(\m)$.
Let $K$ be an $R$-module with $K\cong\omega$ and $R\subseteq K\subseteq\overline R$ in $Q$, where $\overline R$ is the integral closure of $R$.
Using \cite[Proposition 2.5]{O}, we have:
\begin{equation}\label{50}
\m\cong\Hom_R(\m,R)=E\quad\text{and}\quad\m^\dag\cong\Hom_R(\m,K)\cong(K:_Q\m).
\end{equation}
By (4) the module $\m^\dag$ belongs to $\ocms(R)$.
It follows from \cite[Theorem 2.14]{Ko} that $R$ is almost Gorenstein; note that the completion of $R$ also has Gorenstein punctured spectrum by (4').

Let $d>1$.
Since $(\syz^d k)^\dag\in\ocm(R)$, there is an exact sequence $0 \to (\syz^dk)^\dag \to R^{\oplus m} \to N \to0$ for some $m\ge0$ and $N\in\cm(R)$.
Choose a parameter ideal $Q=(x_1,\dots,x_d)$ of $R$ satisfying the equality $\m^2=Q\m$, and set $\overline{(-)}=(-)\otimes_RR/(x_1)$.
An exact sequence
$$
0 \to \overline{(\syz^d k)^\dag}\to \overline{R}^{\oplus m} \to \overline N \to 0
$$
is induced, which shows that $\overline{(\syz^d k)^\dag}$ is in $\ocm(\overline R)$.
Applying $(-)^\dag$ to the exact sequence $0\to\syz^dk\xrightarrow{x}\syz^dk\to\overline{\syz^dk}\to0$ and using \cite[Lemma 3.1.16]{BH}, we obtain isomorphisms
$$
\overline{(\syz^dk)^\dag}\cong\Ext_R^1(\overline{\syz^dk},\omega)\cong\Hom_{\overline R}(\overline{\syz^dk},\overline\omega).
$$
The module $\syz_{\overline R}^{d-1}k$ is a direct summand of $\overline{\syz^dk}$ by \cite[Corollary 5.3]{syz2}, and hence $\Hom_{\overline R}(\syz_{\overline R}^{d-1}k,\overline\omega)$ is a direct summand of $\Hom_{\overline R}(\overline{\syz^dk},\overline\omega)$.
Summarizing these, we observe that $\Hom_{\overline R}(\syz_{\overline R}^{d-1}k,\overline\omega)$ belongs to $\ocm(\overline R)$.
Since $\overline R$ has minimal multiplicity, we can apply the induction hypothesis to $\overline R$ to conclude that $\overline R$ is almost Gorenstein, and so is $R$ by \cite[Theorem 3.7]{almgor}.
\end{proof}

\begin{rem}
%Keep the notation of Theorem \ref{five}.
When $d\ge2$, it holds that
$$
\Ext_R^{d+1}(\tr((\syz^dk)^\dag),R)\cong\Ext_R^{d-1}(\Hom_R(\omega,\syz^dk),R).
$$
Thus Theorem \ref{five}(7) can be replaced with the condition that $\Ext_R^{d-1}(\Hom_R(\omega,\syz^dk),R)=0$.

Indeed, using the Hom-$\otimes$ adjointness twice, we get isomorphisms
$$
\Hom_R(\omega,\syz^dk)\cong\Hom_R(\omega,(\syz^dk)^{\dag\dag})\cong\Hom_R((\syz^dk)^\dag\otimes_R\omega,\omega)\cong\Hom_R((\syz^dk)^\dag,\omega^\dag)\cong(\syz^dk)^{\dag *},
$$
and $(\syz^dk)^{\dag *}$ is isomorphic to $\syz^2\tr((\syz^dk)^\dag)$ up to free summand.
\end{rem}

We have several more conditions related to the equality $\ocms(R)=\ul(R)$.

\begin{cor}\label{fouro}
Let $R$ be as in Theorem \ref{five}.
Consider the following conditions:\par
{\rm(1)} $(\syz^dk)^\dag\cong\syz^dk$,\quad
{\rm(2)} $(\syz^dk)^\dag\in\add(\syz^dk)$,\quad
{\rm(3)} $\a(\syz^dk)^\dag=\m$,\quad
{\rm(4)} $\ocms(R)=\ul(R)$.\\
It then holds that {\rm(1)} $\Longrightarrow$ {\rm(2)} $\Longleftrightarrow$ {\rm(3)} $\Longrightarrow$ {\rm(4)}.
\end{cor}

\begin{proof}
The implications (1) $\Rightarrow$ (2) $\Rightarrow$ (3) are obvious.
The proof of Proposition \ref{7b} shows that if an Ulrich $R$-module $M$ satisfies $\a M=\m$, then $M$ is in $\add(\syz^dk)$.
This shows (3) $\Rightarrow$ (2).
Proposition \ref{8}(1) says that $\syz^dk$ is in $\ocms(R)$, and so is $(\syz^d k)^\dag$ by assumption.
Theorem \ref{five} shows (2) $\Rightarrow$ (4).
\end{proof}

We close this section by constructing an example by applying the above corollary.

\begin{ex}
Let $S=\C[[x,y,z]]$ be a formal power series ring.
Let $G$ be the cyclic group $\frac{1}{2}(1,1,1)$, and let $R=S^G$ be the invariant (i.e. the second Veronese) subring of $S$.
Then $\ocms(R)=\ul(R)$.
In fact, by \cite[Proposition (16.10)]{Y}, the modules
$
R,\,\omega,\,\syz\omega
$
are the nonisomorphic indecomposable Cohen--Macaulay $R$-modules and $(\syz\omega)^\dag\cong\syz\omega$.
By \cite[Theorem 4.3]{syz2} the module $\syz^2\C$ does not have a nonzero free or canonical summand.
Hence $\syz^2\C$ is a direct sum of copies of $\syz\omega$, and thus $(\syz^2\C)^\dag\cong\syz^2\C$.
The equality $\ocms(R)=\ul(R)$ follows from Corollary \ref{fouro}.
\end{ex}

%%%%%%%%%%%%%%%%%%%%%%%%%%%%%%%%%%%%%%%%%%%%%%%%%%%%%%%%%%%
\section{Applications}\label{app}

This section is devoted to stating applications of our main theorems obtained in the previous section.

%%%%%%%%%%%%%%%%%%%%%%%%%%%%%%%%%%%%%%%%%%%%%%%%%
\subsection{The case of dimension one}
We begin with studying the case where $R$ has dimension $1$.

\begin{thm} \label{alm}
Let $(R,\m,k)$ be a $1$-dimensional Cohen--Macaulay local ring with $k$ infinite and canonical module $\omega$.
Suppose that $R$ has minimal multiplicity, and set $(-)^\dag=\Hom_R(-,\omega)$.
Then
$$
\ocms(R)=\ul(R)\ \Longleftrightarrow\ 
\m^\dag\in\ocms(R)\ \Longleftrightarrow\ 
\m^\dag\cong\m\ \Longleftrightarrow\ 
\text{$R$ is almost Gorenstein.}
$$
\end{thm}

\begin{proof}
Call the four conditions (i)--(iv) from left to right.
The implications (i) $\iff$ (ii) $\implies$ (iv) are shown by Theorem \ref{five}, while (iii) $\iff$ (iv) by \cite[Theorem 2.14]{Ko} and \eqref{50}.
Lemma \ref{8}(1) shows (iii) $\implies$ (ii).
\end{proof}

Now we pose a question related to Question \ref{1}.

\begin{ques}
Can we classify $1$-dimensional Cohen--Macaulay local rings $R$ with minimal multiplicity (and infinite residue field) satisfying the condition $\#\ind\ul(R)<\infty$?
\end{ques}

According to Proposition \ref{2}, over such a ring $R$ we have the property that $\#\ind\ocm(R)<\infty$, which is studied in \cite{Ko2}.
If $R$ has finite Cohen--Macaulay representation type (that is, if $\#\ind\cm(R)<\infty$), then of course this question is affirmative.
However, we do not have any partial answer other than this.
The reader may wonder if the condition $\#\ind\ul(R)<\infty$ implies the equality $\ocms(R)=\ul(R)$.
Using the above theorem, we observe that this does not necessarily hold:

\begin{ex}
Let $R=k[[t^3,t^7,t^8]]$ be (the completion of) a numerical semigroup ring, where $k$ is an algebraically closed field of characteristic zero.
Then $R$ is a Cohen--Macaulay local ring of dimension $1$ with minimal multiplicity.
%We have
%$$
%R[\frac{\m}{t^3}]=R[\frac{t^7}{t^3},\frac{t^8}{t^3}]=R[t^4,t^5]=k[[t^3,t^4,t^5]],
%$$
%and the last term has finite representation type by \cite[Page 69]{Y}.
It follows from \cite[Theorem A.3]{HUB} that $\#\ind\ul(R)<\infty$.
On the other hand, $R$ is not almost-Gorenstein by \cite[Example 4.3]{GMP}, so $\ocms(R)\ne\ul(R)$ by Theorem \ref{alm}.
\end{ex}

%%%%%%%%%%%%%%%%%%%%%%%%%%%%%%%%%%%%%%%%%%%
\subsection{The case of dimension two}

From now on, we consider the case where $R$ has dimension $2$.
We recall the definition of a Cohen--Macaulay approximation.
Let $R$ be a Cohen--Macaulay local ring with canonical module.
A homomorphism $f:X\to M$ of $R$-modules is called a {\em Cohen--Macaulay approximation} (of $M$) if $X$ is Cohen--Macaulay and any homomorphism $f':X'\to M$ with $X'$ being Cohen--Macaulay factors through $f$.
It is known that $f$ is a (resp. minimal) Cohen--Macaulay approximation if and only if there exists an exact sequence
$$
0 \to Y \xrightarrow{g} X \xrightarrow{f} M \to 0
$$
of $R$-modules such that $X$ is Cohen--Macaulay and $Y$ has finite injective dimension (resp. and that $X,Y$ have no common direct summand along $g$).
For details of Cohen--Macaulay approximations, we refer the reader to \cite[Chapter 11]{LW}.

The module $E$ appearing in the following remark is called the {\em fundamental module} of $R$.

\begin{rem}\label{fr}
Let $(R,\m,k)$ be a $2$-dimensional Cohen--Macaulay local ring with canonical module $\omega$.
\begin{enumerate}[(1)]
\item
There exists a nonsplit exact sequence
\begin{equation}\label{fs}
0 \to \omega \to E \to \m \to 0
\end{equation}
which is unique up to isomorphism.
This is because $\Ext_R^1(\m,\omega)\cong\Ext_R^2(k,\omega)\cong k$.
\item
The module $E$ is Cohen--Macaulay and uniquely determined up to isomorphism.
\item
The sequence \eqref{fs} gives a minimal Cohen--Macaulay approximation of $\m$.
\item
There is an isomorphism $E\cong E^\dag$.
In fact, applying $(-)^\dag$ to \eqref{fs} induces an exact sequence
$$
0 \to \m^\dag\to E^\dag\to R\to\Ext_R^1(\m,\omega)\to\Ext_R^1(E,\omega)=0.
$$
Applying $(-)^\dag$ to the natural exact sequence $0\to\m\to R\to k\to0$ yields $\m^\dag\cong\omega$, while $\Ext_R^1(\m,\omega)\cong k$.
We get an exact sequence $0\to\omega\to E^\dag\to\m\to0$, and the uniqueness of \eqref{fs} shows $E^\dag\cong E$.
\end{enumerate}
\end{rem}

To prove the main result of this section, we prepare two lemmas.
The first one relates the fundamental module of a $2$-dimensional Cohen--Macaulay local ring $R$ with $\ul(R)$ and $\ocms(R)$.

\begin{lem} \label{4ff}
Let $(R,\m,k)$ be a $2$-dimensional Cohen--Macaulay local ring with canonical module $\omega$ and fundamental module $E$.
\begin{enumerate}[\rm(1)]
\item
Assume that $R$ has minimal multiplicity.
Then $E$ is an Ulrich $R$-module.
\item
For each module $M\in\ocms(R)$ there exists an exact sequence
$
0 \to M \to E^{\oplus n} \to N \to 0
$
of $R$-modules such that $N$ is Cohen--Macaulay.
\end{enumerate}
\end{lem}

\begin{proof}
(1) We may assume that $k$ is infinite by Remark \ref{infin}(2).
Let $Q=(x,y)$ be a parameter ideal of $R$ with $\m^2=Q\m$.
We have $\m/x\m\cong\m/(x)\oplus k$; see \cite[Corollary 5.3]{syz2}.
Note that $(\m/(x))^2=y(\m/(x))$.
By \cite[Corollary 2.5]{K} the minimal Cohen--Macaulay approximation of $\m/x\m$ as an $R/(x)$-module is $E/xE$.
In view of the proof of \cite[Proposition 11.15]{LW}, the minimal Cohen--Macaulay approximations of $\m/(x)$ and $k$ as $R/(x)$-modules are $\m/(x)$ and $\Hom_{R/(x)}(\m/(x),\omega/x\omega)$, respectively.
Thus we get an isomorphism
$$
E/xE\cong\m/(x)\oplus\Hom_{R/(x)}(\m/(x),\omega/x\omega).
$$
In particular, $E/xE$ is an Ulrich $R/(x)$-module by Lemma \ref{8}(1) and Corollary \ref{5.6}.
It follows from Lemma \ref{7a} that $E$ is an Ulrich $R$-module.

(2) Take an exact sequence $0\to M \xrightarrow{f} R^{\oplus n} \xrightarrow{e} L \to 0$ such that $L$ is Cohen--Macaulay.
As $M$ has no free summand, the homomorphism $e$ is minimal.
This means that $f$ factors through the natural inclusion $i:\m^{\oplus n} \to R^{\oplus n}$, that is, $f=ig$ for some $g\in\Hom_R(M,\m^{\oplus n})$.
The direct sum $p:E^{\oplus n} \to \m^{\oplus n}$ of copies of the surjection $E\to\m$ (given by \eqref{fs}) is a Cohen--Macaulay approximation.
Hence there is a homomorphism $h:M\to E^{\oplus n}$ such that $g=ph$.
We get a commutative diagram
$$
\xymatrix@R-1pc{
0 \ar[r] & M \ar[r]^f  & R^{\oplus n} \ar[r] & L \ar[r] & 0\\
0 \ar[r] & M \ar@{=}[u] \ar[r]^h & E^{\oplus n} \ar[r] \ar[u]^{ip} & N \ar[r] \ar[u] & 0
}
$$
with exact rows.
This induces an exact sequence $0 \to E^{\oplus n} \to R^{\oplus n}\oplus N \to L \to 0$, and therefore $N$ is a Cohen--Macaulay $R$-module.
\end{proof}

A short exact sequence of Ulrich modules is preserved by certain functors:

\begin{lem} \label{4fff}
Let $0\to X \to Y \to Z \to 0$ be an exact sequence of modules in $\ul(R)$.
Then it induces exact sequences of $R$-modules
\begin{enumerate}[\qquad\rm(a)]
\item
$0 \to X\otimes_R k \to Y\otimes_R k \to Z\otimes_R k \to 0$,
\item
$0 \to \Hom_R(Z,k)\to \Hom_R(Y,k) \to \Hom_R(X,k) \to 0$, and
\item
$0 \to \Hom_R(Z,(\syz^d k)^\dag) \to \Hom_R(Y,(\syz^d k)^\dag) \to \Hom_R(X,(\syz^d k)^\dag) \to 0$.
\end{enumerate}
\end{lem}

\begin{proof}
The sequence $X\otimes_R k \to Y\otimes_R k \to Z\otimes_R k \to 0$ is exact and the first map is injective by Lemma \ref{7.9}(1).
Hence (a) is exact, and so is (b) by a dual argument.
In what follows, we show that (c) is exact.
We first note that $(\syz^d k)^\dag$ is a minimal Cohen--Macaulay approximation of $k$; see the proof of \cite[Proposition 11.15]{LW}.
Thus there is an exact sequence $0 \to I \to (\syz^d k)^\dag \to k \to 0$ such that $I$ has finite injective dimension.
As $\ul(R)\subseteq\cm(R)$, we have $\Ext_R^1(M,I)=0$ for all $M\in\{X,Y,Z\}$.
We obtain a commutative diagram
$$
\xymatrix@R-1pc{
0 \ar[r] & \Hom_R(Y,I) \ar@{->>}[d]^\alpha \ar[r] & \Hom_R(Y, (\syz^d k)^\dag) \ar[d]^\beta \ar[r] & \Hom_R(Y,k) \ar[d]^\gamma\ar[r] & 0\\
0 \ar[r] & \Hom_R(X,I)  \ar[r] & \Hom_R(X, (\syz^d k)^\dag) \ar[r] & \Hom_R(X,k) \ar[r] & 0}
$$
with exact rows, where $\alpha$ is surjective.
The exactness of (b) implies that $\gamma$ is surjective.
By the snake lemma $\beta$ is also surjective, and therefore (c) is exact.
\end{proof}

Now we can state and show our main result in this section.

\begin{thm}\label{four}
Let $R$ be a $2$-dimensional complete singular normal local ring with residue field $\C$ and having minimal multiplicity.
Suppose that $R$ does not have a cyclic quotient singularity.
Then:
$$
(\syz^dk)^\dag\cong\syz^dk\ \Longleftrightarrow\ 
(\syz^dk)^\dag\in\add(\syz^dk)\ \Longleftrightarrow\ 
\a(\syz^dk)^\dag=\m\ \Longleftrightarrow\ 
\ocms(R)=\ul(R).
$$
\end{thm}

\begin{proof}
In view of Corollary \ref{fouro}, it suffices to show that if $R$ does not have a cyclic quotient singularity, then the fourth condition implies the first one.
By virtue of \cite[Theorem 11.12]{Y} the fundamental module $E$ is indecomposable.
Applying Lemma \ref{4ff}(2) to $(\syz^d k)^\dag$, we have an exact sequence
$
0 \to (\syz^d k)^\dag \xrightarrow{\alpha} E^{\oplus n} \to N \to 0
$
such that $N$ is Cohen--Macaulay.
Since $E$ is Ulrich by Lemma \ref{4ff}(1), so are all the three modules in this sequence by Lemma \ref{7.9}(2).
Thus we can apply Lemma \ref{4fff} to see that the induced map
$$
\Hom_R(\alpha,(\syz^dk)^\dag):\Hom_R(E^{\oplus n},(\syz^d k)^\dag)\to\Hom_R((\syz^d k)^\dag,(\syz^d k)^\dag) 
$$
is surjective.
This implies that $\alpha$ is a split monomorphism, and $(\syz^d k)^\dag$ is isomorphic to a direct summand of $E^{\oplus n}$.
Since $E$ is indecomposable, it folllows that $(\syz^d k)^\dag$ is isomorphic to $E^{\oplus m}$ for some $m$.
We obtain
$$
(\syz^d k)^\dag\cong E^{\oplus m} \cong (E^\dag)^{\oplus m}\cong (\syz^d k)^{\dag\dag}\cong \syz^d k,
$$
where the second isomorphism follows by Remark \ref{fr}(4).
\end{proof}

\begin{rem}\label{ny}
Let $R$ be a cyclic quotient surface singularity over $\C$.
Nakajima and Yoshida \cite[Theorem 4.5]{NY} give a necessary and sufficient condition for $R$ to be such that the number of nonisomorphic indecomposable Ulrich $R$-modules is equal to the number of nonisomorphic nonfree indecomposable special Cohen--Macaulay $R$-modules.
By \cite[Corollary 2.9]{IW}, the latter is equal to the number of isomorphism classes of indecomposable modules in $\ocms(R)$.
Therefore, they actually gives a necessary and sufficient condition for $R$ to satisfy $\ocms(R)=\ul(R)$.
\end{rem}

Using our Theorem \ref{four}, we give some examples of a quotient surface singularity over $\C$ to consider Ulrich modules over them.

\begin{ex}
(1) Let $S=\C[[x,y]]$ be a formal power series ring.
Let $G$ be the cyclic group $\frac{1}{3}(1,1)$, and let $R=S^G$ be the invariant (i.e. the third Veronese) subring of $S$.
Then $\ocms(R)=\ul(R)$.
This follows from \cite[Theorem 4.5]{NY} and Remark \ref{ny}, but we can also show it by direct caluculation: we have
$$
\Cl(R)=\{[R],[\omega],[\p]\}\cong\Z/3\Z,
$$
where $\omega=(x^3,x^2y)R$ is a canonical ideal of $R$, and $\p=(x^3,x^2y,xy^2)R$ is a prime ideal of height $1$ with $[\omega]=2[\p]$.
Since the second Betti number of $\C$ over $R$ is $9$, we see $\syz^2\C\cong\p^{\oplus3}$.
As $[\p^\dag]=[\omega]-[\p]=[\p]$, we have $\p^\dag\cong\p$ and $(\syz^2\C)^\dag\cong\syz^2\C$. 
Theorem \ref{four} shows $\ocms(R)=\ul(R)$.

(2) Let $S=\C[[x,y]]$ be a formal power series ring.
Let $G$ be the cyclic group $\frac{1}{8}(1,5)$, and let $R=S^G$ be the invariant subring of $S$.
With the notation of \cite{NY}, the Hirzebruch-Jung continued fraction of this group is $[2,3,2]$. It follows from \cite[Theorem 4.5]{NY} and Remark \ref{ny} that $\ocms(R)\not=\ul(R)$.
\end{ex}

%%%%%%%%%%%%%%%%%%%%%%%%%%%%%%%%%%%
\subsection{An exact structure of the category of Ulrich modules}

Finally, we consider realization of the additive category $\ul(R)$ as an exact category in the sense of Quillen \cite{Q}.
We begin with recalling the definition of an exact category given in \cite[Appendix A]{Ke}.

\begin{dfn}\label{ex}
Let $\mathcal{A}$ be an additive category.
A pair $(i,d)$ of composable morphisms
$$
X\xrightarrow[]{i} Y \xrightarrow[]{d} Z
$$
is {\it exact} if $i$ is the kernel of $d$ and $d$ is the cokernel of $i$.
Let $\mathcal{E}$ be a class of exact pairs closed under isomorphism.
The pair $(\mathcal{A},\mathcal{E})$ is called an {\it exact category} if the following axioms hold.
Here, for each $(i,d)\in\mathcal{E}$ the morphisms $i$ and $d$ are called an {\it inflation} and a {\it deflation}, respectively.
\begin{itemize}
\item[(Ex0)]
$1:0\to 0$ is a deflation.
\item[(Ex1)]
The composition of deflations is a deflation.
\item[(Ex2)]
For each morphism $f:Z' \to Z$ and each deflation $d:Y \to Z$, there is a pullback diagram as in the left below, where $d'$ is a deflation.
\item[(Ex2$^{\rm op}$)]
For each morphism $f:X\to X'$ and each inflation $i:X\to Y$, there is a pushout diagram as in the right below, where $i'$ is an inflation.
$$
\xymatrix@R-1pc{
Y' \ar[d] \ar[r]^{d'} & Z' \ar[d]^f \\
Y \ar[r]^d & Z
}
\qquad\qquad\qquad
\xymatrix@R-1pc{
X \ar[r]^i \ar[d]_f & Y \ar[d]\\
X' \ar[r]^{i'} & Y'
}
$$
\end{itemize}
\end{dfn}

We can equip a structure of an exact category with our $\ul(R)$ as follows.

\begin{thm}
Let $R$ be a $d$-dimensional Cohen--Macaulay local ring with residue field $k$ and canonical module, and assume that $R$ has minimal multiplicity.
Let $\s$ be the class of exact sequences $0 \to L \to M \to N \to 0$ of $R$-modules with $L,M,N$ Ulrich.
Then $\ul(R)=(\ul(R),\s)$ is an exact category having enough projective objects and enough injective objects with $\proj\ul(R)=\add(\syz^dk)$ and $\inj\ul(R)=\add((\syz^dk)^\dag)$.
\end{thm}

\begin{proof}
We verify the axioms in Definition \ref{ex}.

(Ex0):
This is clear.

(Ex1):
Let $d:Y\to Z$ and $d':Z \to W$ be deflations.
Then there is an exact sequence
$
0 \to X \to Y \xrightarrow{d'd} W \to 0
$
of $R$-modules.
Since $Y$ is in $\ul(R)$ and $X,W\in \cm(R)$, it follows from  that $X\in\ul(R)$.
Thus this sequence belongs to $\s$, and $d'd$ is a deflation.

(Ex2):
Let $f:Z' \to Z$ be a homomorphism in $\ul(R)$ and $d:Y \to Z$ a deflation in $\s$.
Then we get an exact sequence
$
0\to Y' \to Y\oplus Z' \xrightarrow{(d,f)} Z \to 0.
$
Since $Y\oplus Z'\in \ul(R)$ and $Y',Z\in\cm(R)$, Lemma \ref{7.9}(2) implies $Y'\in \ul(R)$.
Make an exact sequence $0 \to X' \to Y' \xrightarrow{d'} Z' \to 0$.
As $Y'\in \ul(R)$ and $X',Z'\in\cm(R)$, the module $Z'$ is in $\ul(R)$ by Lemma \ref{7.9}(2) again.
Thus $d'$ is a deflation.

(Ex2$^{\rm op}$):
We can check this axiom by the opposite argument to (Ex2). 

Now we conclude that $(\ul(R),\s)$ is an exact category.
Let us prove the remaining assertions.
Lemma \ref{4fff}(c) yields the injectivity of $(\syz^d k)^\dag$.
Since $(-)^\dag$ gives an exact duality of $(\ul(R),\s)$, the module $\syz^d k$ is a projective object.
We also observe from Lemma \ref{8} and Corollary \ref{5.6} that $(\ul(R),\s)$ has enough projective objects with $\proj\ul(R)=\add(\syz^dk)$, and has enough injective objects with $\inj\ul(R)=\add((\syz^dk)^\dag)$ by the duality $(-)^\dag$.
\end{proof}

\begin{rem}
Let $(R,\m)$ be $1$-dimensional Cohen--Macaulay local ring with infinite residue field.
Let $(t)$ be a minimal reduction of $\m$.
Then
$
\ul(R)=\cm\left(R\left[\frac{\m}{t}\right]\right)
$
by \cite[Proposition A.1]{HUB}.
This equality acturely gives an equivalence $\ul(R)\cong\cm(R[\frac{\m}{t}])$ of categories, since Hom-sets do not change; see \cite[Proposition 4.14]{LW}.
Thus the usual exact structure on $\cm(R[\frac{\m}{t}])$ coincides with the exact structure on $\ul(R)$ given above via this equivalence.
\end{rem}

%%%%%%%%%%%%%%%%%%%%%%%%%%%%%%%%%%
\section*{Acknowledgments}

The authors are grateful to Doan Trung Cuong for valuable discussion on Ulrich modules.
In particular, his Question \ref{1} has given a strong motivation for them to have this paper.

%%%%%%%%%%%%%%%%%%%%%%%%%%%%%%%%%%%%%%%%%%%%%%%%%

\end{document}